\documentclass[11pt]{article}%
\usepackage{amssymb}
\usepackage{amsfonts}
\usepackage{amsmath}
\usepackage{graphicx}%
\providecommand{\U}[1]{\protect\rule{.1in}{.1in}}
\newtheorem{theorem}{Theorem}

\newtheorem{definition}[theorem]{Definition}

\newtheorem{lemma}[theorem]{Lemma}

\newtheorem{remark}[theorem]{Remark}

\newenvironment{proof}[1][Proof]{\noindent\textbf{#1.} }{\ \rule{0.5em}{0.5em}}
\begin{document}

\title{Prime order automorphisms of Klein surfaces representable by rotations of the
euclidean space}
\author{Antonio F. Costa\\UNED\\Facultad de Ciencias\\Senda del Rey, 9\\28040 Madrid\\Spain
\and Cam Van Quach Hongler\\Universit\'{e} de Gen\`{e}ve\\Section de mathematiques\\2-4, rue du Li\`{e}vre\\CH-1211 Gen\`{e}ve 64\\Switzerland}
\maketitle

\section{Introduction}

A (closed) surface embedded in the Euclidean space inherits a conformal
structure so it is a Riemann surface. This is the easiest way to present
examples of Riemann surfaces. In 1882, F. Klein asked the following question
(\cite{K1}): is every Riemann surface conformally equivalent to an embedded
surface in the Euclidean space?. The positive answer was given by A. Garsia in
1960, \cite{G}. Several generalizations have been produced in \cite{R1} and
\cite{Ko}, \cite{Ko2}, \cite{Ko3}.

In the study of moduli spaces of Riemann surfaces, the surfaces with
automorphisms play an important r\^{o}le: the branch loci for the orbifold
structure of the moduli space consists of such type of surfaces. A classical
and easy method to present examples of Riemann surfaces with automorphisms is
to consider embedded surfaces in the Euclidean space that are invariant by
isometries. R. R\"{u}edy in \cite{R2} describes the automorphisms of compact
Riemann surfaces that can be represented by rotations of the space acting on
embedded surfaces. Later, in \cite{C}, the anticonformal automorphisms of
Riemann surfaces representable by isometries on embedded Riemann surfaces were
completely characterized. Now we are interested in the presentation and
visualization of automorphisms of Riemann surfaces with boundary as embedded surfaces.

In order to consider bordered surfaces we shall use the theory of Klein
surfaces. A compact\ bordered (with non-empty boundary) orientable Klein
surface is a compact bordered and orientable surface endowed with a dianalytic
structure. Let $S$ be a bordered orientable Klein surface and $p$ a prime.
Assume that $f$ is an order $p$ automorphism of $S$. In this work we obtain
the conditions on the topological type of $(S,f)$ to be conformally equivalent
to $(S^{\prime},f^{\prime})$ where $S^{\prime}$ is a bordered orientable Klein
surface embedded in the Euclidean space and $f^{\prime}$ is the restriction to
$S^{\prime}$ of a prime order rotation. The boundary of $S^{\prime}$ is a link
in $\mathbb{R}^{3}$ and $S^{\prime}$ is a Seifert surface of $\partial
S^{\prime}$. Since the Klein structure follows from a result in \cite{Ko}, the
main difficulty for our task is the topological realization (Lemma 5).

Our results can help to visualize some automorphisms of Riemann surfaces that
are not easy to see representing such automorphisms as restrictions of
isometries of $\mathbb{S}^{4}$ or $\mathbb{R}^{4}$. In Section 3 we represent
two famous automorphisms using this method: the order seven automorphisms of
the Klein quartic and the Wiman surface.

\section{Prime order automorphisms of Klein surfaces representable by
rotations}

A bordered orientable Klein surface is a compact surface with non-empty
boundary endowed with a dianalytic structure, i.e. an equivalence class of
dianalytic atlases. A dianalytic atlas is a set of complex charts with
analytic or antianalytic transition maps (see \cite{AG}). If $S$ is an
orientable Klein surface with genus $g$ and $k$ boundary components, there is
a complex double $\widehat{S}$ of $S$ that is a Riemann surface with genus
$\widehat{g}=2g+k-1$. An automorphism $f$ of $S$ is an (auto)-homeomorphism of
$S$ such that locally on each chart of $S$ is analytic or antianalytic. If
$\widehat{g}>1$ then the automorphism $f$ has finite order. Thus we shall
assume $\widehat{g}>1$ and the orientability of the surfaces.

Let $S$ and $S^{\prime}$ be bordered orientable Klein surfaces that are
homeomorphic and let $f,f^{\prime}$ be automorphisms of $S$ and $S^{\prime} $
respectively. We say that $f$ and $f^{\prime}$ are topologically equivalent or
have the same topological type if there is a homeomorphism $h:S\rightarrow
S^{\prime}$, such that $f=h^{-1}\circ f^{\prime}\circ h$.

Two automorphims topologically equivalent have the same order and the same
number of fixed points. To avoid technical complications we shall consider
automorphisms of prime order $p>2$. In order to describe the topological types
of automorphisms of bordered orientable Klein surfaces we need two sets of invariants.

\begin{definition}
Rotation indexes for the fixed points of an automorphism. Let $S$ be a
bordered orientable Klein surface and $f$ be an automorphism of $S$ of prime
order $p$. Assume $S$ is oriented. Let $\{x_{1},...,x_{r}\}$ be the set of
fixed points of $f$. Let $D_{1}$,...,$D_{r}$ be a set of disjoint discs in $S$
such that $x_{j}\in D_{j}$. There are charts $c_{j}:D_{j}\rightarrow
\mathbb{C}$ of $S$ (sending the orientation of $D_{j}$ on the canonical
orientation of $\mathbb{C}$) such that $c_{j}\circ f\circ c_{j}^{-1}%
:c_{j}(D_{j})\rightarrow c_{j}(D_{j})$ is $z\rightarrow e^{2\pi iq_{j}/p}$,
$-p/2<q_{j}<p/2$. The set of rotation indexes $\{\epsilon2q_{1}\pi
/p,...,\epsilon2q_{r}\pi/p\}$, with $\epsilon=1$ or $-1$, is an invariant for
the topological type of $f$.
\end{definition}

\begin{definition}
Rotation indexes for the boundary components. Let $S$ be a bordered orientable
Klein surface and $f$ be an automorphism of $S$ of order $p$. Let
$\{C_{1},...,C_{k}\}$ be the set of boundary connected components of $S$.
Assume that $S$ is oriented. There are homeomorphisms $h_{j}:C_{j}%
\rightarrow\mathbb{S}^{1}=\{z\in\mathbb{C}:\left\Vert z\right\Vert =1\}$
preserving the orientation of $C_{j}$ induced by the orientation of $S$ such
that $h_{j}\circ f\circ h_{j}^{-1}:h_{j}(D_{j})\rightarrow h_{j}(D_{j})$ is
$z\rightarrow e^{2\pi iq_{j}/p}$, $-p/2<q_{j}<p/2$. The set of rotation
indexes $\{\epsilon2q_{1}\pi/p,...,\epsilon2q_{k}\pi/p\}$, with $\epsilon=1$
or $-1$, is an invariant for the topological type of $f$.
\end{definition}

The fact of $S$ be a Klein surface does not give an orientation to $S$ then
the sets $\{2q_{1}\pi/p,...,2q_{k}\pi/p\}$ and $\{-2q_{1}\pi/p,...,-2q_{k}%
\pi/p\}$ must be considered equal.

\begin{theorem}
[Classification of automorphisms of Klein surfaces \cite{Y}]Let $f$ and
$f^{\prime}$ two orientation preserving automorphisms of prime order $p$ of
two homeomorphic bordered orientable Klein surfaces. The automorphisms $f$ and
$f^{\prime}$ are topologically equivalent if and only if the set of rotation
indexes for the fixed points and the boundary components of $f$ are equal to
the set of rotation indexes for the fixed points and the boundary components
of $f^{\prime}$.
\end{theorem}

Let $S$ be a bordered orientable surface embedded in the Euclidean space. The
Euclidean metric induces a structure of Klein surface on $S$. If $\rho$ is a
rotation in the space and $S$ is invariant by $\rho$, $\rho$ restricted to $S$
is an automorphism $\rho_{S}$ of the Klein surface $S$. If $h:S\rightarrow K$
is an isomorphism between Klein surfaces the automorphism $h\circ\rho_{S}\circ
h^{-1}$ is an \textit{automorphism representable as a rotation}.

\begin{theorem}
Let $f$ be an automorphism of prime order $p$ of a bordered orientable Klein
surface (we consider that $f$ is orientation preserving if $p=2$). The
automorphism $f$ is representable as a rotation if and only if the set of
rotation indices of fixed points of $f$ is contained in $\{\pm2q\pi/p\}$, for
some $q$ such that $0<q<p/2$.
\end{theorem}

Note that there are no restriction on the rotation indices of boundary components.

If $f$ is representable as a rotation of angle $2q\pi/p$ then it is necessary
that the rotation indices of the fixed points of $f$ are $2q\pi/p$ or
$-2q\pi/p$. For the sufficiency in the Theorem first we need a topological
construction that we shall present in the following Lemma.

\begin{lemma}
[Topological construction]Let $S$ be a bordered orientable compact surface and
$h$ be an autohomeomorphism of $S$ with order $p$ where $p$ is a prime (if
$p=2$ we assume also that $h$ is orientation preserving). Assume that the
rotation index of every fixed point of $h$ is equal to $\pm\frac{2\pi q}{p}$,
where $0<q<q/2$. Hence there is a surface $\Sigma$ embedded in $\mathbb{R}%
^{3}$ such that $\Sigma$ is equivariant under a rotation $\rho$ of angle
$\frac{2\pi q}{p}$ and $(S,h)$ is topologically equivalent to $(\Sigma,\rho)$.
\end{lemma}

\begin{proof}
Consider $S/\left\langle h\right\rangle $ the orbit surface. Assume that
$S/\left\langle h\right\rangle $ has genus $g$ and $t+l$ boundary components
and the quotient orbifold $O(S/\left\langle h\right\rangle )$ has $r$ conic
points. Let $\omega:\pi_{1}O(S/\left\langle h\right\rangle )\rightarrow
C_{p}=\left\langle \gamma\right\rangle $ be the monodromy epimorphism of the
branched covering $S\rightarrow S/\left\langle h\right\rangle $. Corresponding
to $\omega$, there is a canonical presentation of $\pi_{1}O(S/\left\langle
h\right\rangle )$:%
\[
\pi_{1}O(S/\left\langle h\right\rangle )=\left\langle a_{i},b_{i},x_{i}%
,e_{i},c_{i}:\Pi\lbrack a_{i},b_{i}]\Pi x_{i}\Pi e_{i}=1,x_{i}^{p}%
=1,e_{i}c_{i}e_{i}^{-1}c_{i}=1,c_{i}^{2}=1\right\rangle
\]
such that: $\omega(a_{i})=\omega(b_{i})=\omega(c_{i})=1$, $\omega
(x_{j})=\gamma^{q}$, $j=1,...,n$, $\omega(x_{j})=\gamma^{p-q}$, $j=n+1,...,r$,
$\omega(e_{j})=\gamma^{q_{j}}$, $j=1,...,t$, $0<q_{i}<p,$ $\omega(e_{j})=1$,
$j=t+1,...,t+l$. That means there are $n$ fixed points of $h$ with rotation
index $\frac{2\pi q}{p}$ and $r-n$ fixed points of $h$ with rotation index
$-\frac{2\pi q}{p}$ and the rotation index of the boundary component $C_{i}$
of $S/\left\langle h\right\rangle $ corresponding to $e_{i}$ is $\frac{2\pi
q_{i}}{p}$.

Consider the usual orientation of $\mathbb{R}^{3}$ with $Oz$ oriented in the
increasing way of $z$. Consider $r$ discs $D_{i},i=1,...r$
\[
D_{i}=\left\{  (x,y,i):\left\Vert (x,y)\right\Vert \leq1\right\}
\]
Assume that $D_{i}$ has the orientation producing in $\partial D_{i}$ an
orientation such that the linking number with $Oz$ is:
\begin{align*}
Lk(\partial D_{i},Oz)  &  =1,i=1,...,n\\
Lk(\partial D_{i},Oz)  &  =-1,i=n+1,...,r
\end{align*}

Consider in
\[
R_{m}=\{(x,y,z)\in\mathbb{R}^{3}:-m-1<z<-m\}
\]
for $m\in\{1,...,t-1\}$, an embedded oriented annulus $C_{m}$ such that
$C_{m}\cap Oz=\varnothing$, $\partial C_{m}=C_{m}^{+}\cup C_{m}^{-}$ and
\[
Lk(C_{m}^{-},Oz)=q_{m}%
\]
%

\begin{figure}
[ptb]
\begin{center}
\includegraphics[
height=4.8006in,
width=3.7861in
]%
{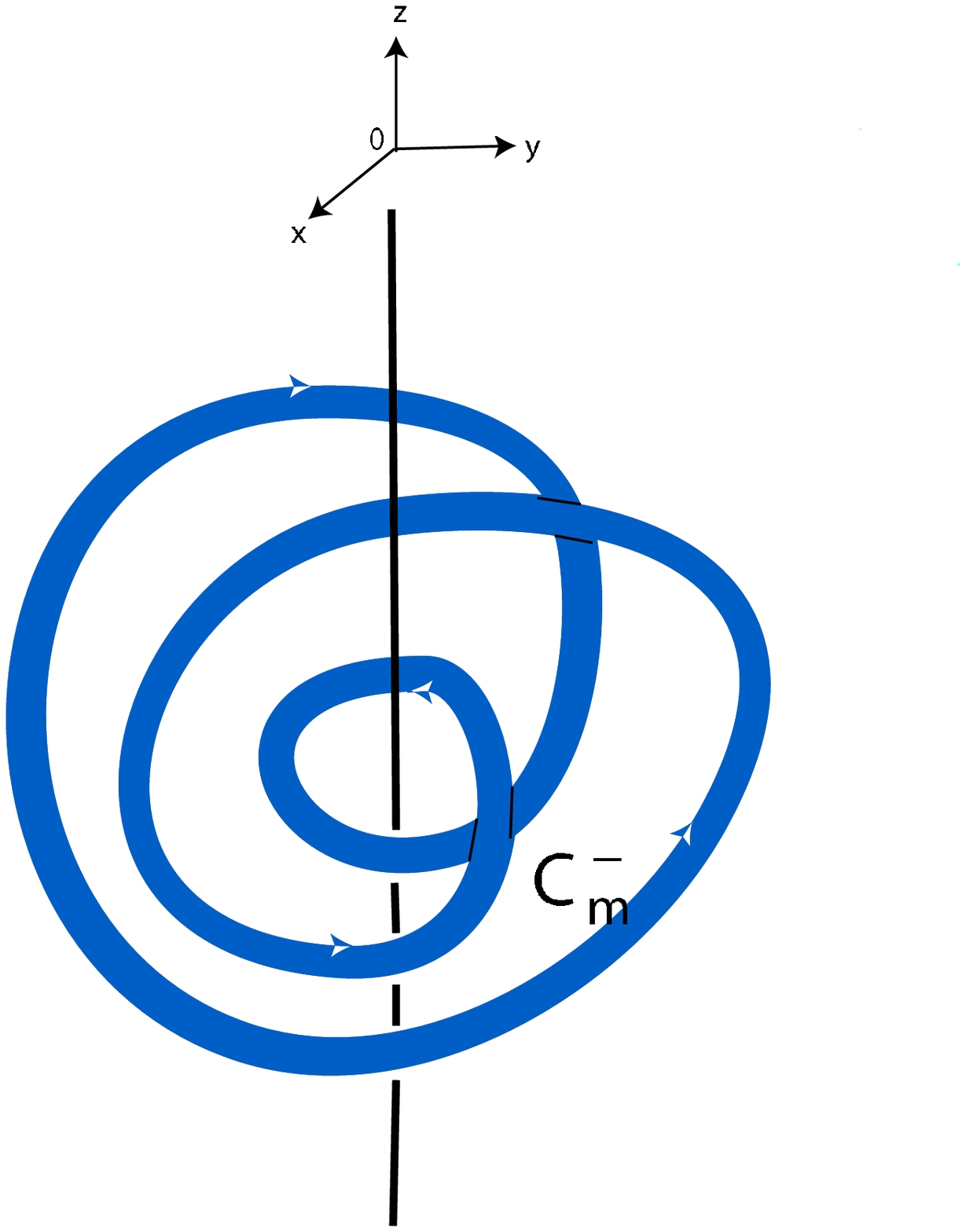}%
\caption{An annulus $C_{m}$ with $q_{m}=3$}%
\end{center}
\end{figure}

Let us connect $D_{i}$ with $D_{i+1}$ by 1-handles $b_{i}$, $C_{i}$ with
$C_{i+1}$ by 1-handles $h_{i}$ (more precisely, $C_{i}^{+}$ to $C_{i+1}^{+}$)
and $D_{1}$ with $C_{1}$ by the 1-handle $h$ (more precisely, $\partial D_{1}$
with $C_{1}^{+}$) in such a way that:%
\[
S_{0}=%
{\textstyle\bigsqcup}
D_{i},C_{i},b_{i},h_{i},h
\]
is an embedded surface homeomorphic to a sphere with $t$ holes. Furthermore
the 1-handles must connect the discs and the annuli such that $S_{0}$ has an
orientation which induces on $D_{i}$ and $C_{i}$ their desired orientations
and $S_{0}\cap Oz$ are exactly the centers of the discs $D_{i}$.%

\begin{figure}
[ptb]
\begin{center}
\includegraphics[
height=3.3027in,
width=1.593in
]%
{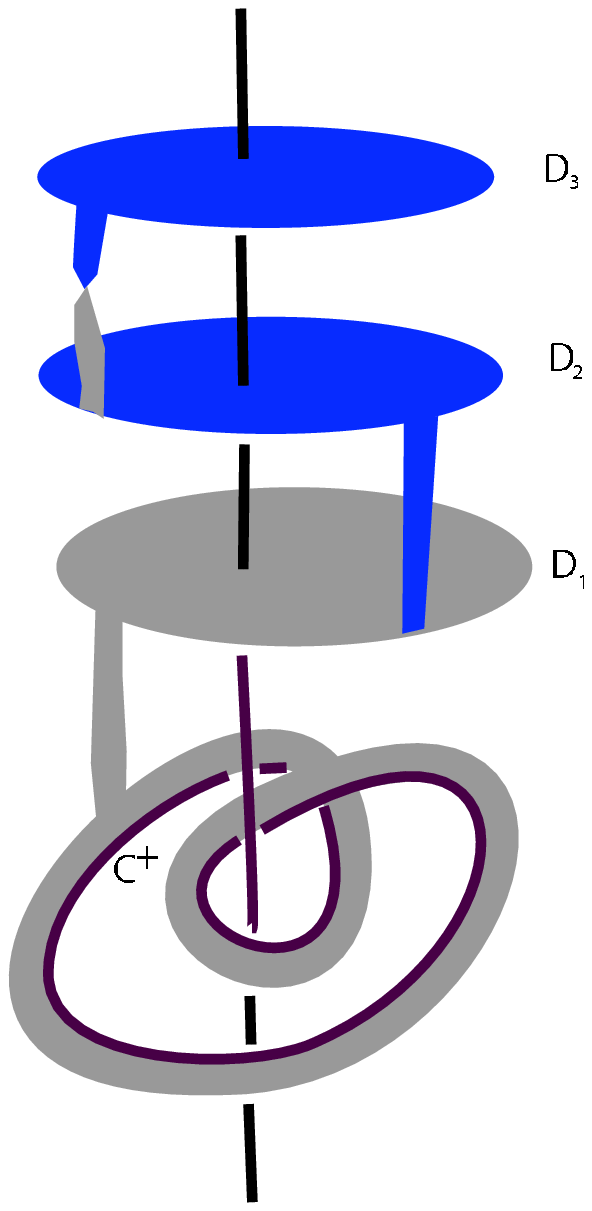}%
\caption{Example of $S_{0}$}%
\end{center}
\end{figure}

Let $F$ be a surface with the topological type of genus $g$ with $l$ boundary
components and contained in
\[
\{(x,y,z)\in\mathbb{R}^{3}:x>2\text{, }z<-t\}
\]
and $Y$ an annulus with the boundaries on the boundaries of discs $\Delta_{1}$
in $S_{0}$ and $\Delta_{2}$ in $F$, and with $Y\cap Oz=\varnothing$. The
surface $(S_{0}-\Delta_{1})\sqcup Y\sqcup(F-\Delta_{1})$ will be denoted $M$.
Let $f:\mathbb{R}^{3}\rightarrow\mathbb{R}^{3}/\left\langle \rho\right\rangle
=\mathbb{R}^{3}$ be the quotient map given by the orbits of the rotation
$\rho$ of angle $\frac{2\pi q}{p}$ around the axis $Oz$. The surface $\Sigma$
that we are looking for is $f^{-1}(M)$. Note that each boundary component
$C_{j}^{-}$ of $C_{j}$ remains as a boundary component of $M$, and since
$Lk(C_{j}^{-},Oz)=q_{j}$, the action of the rotation $\rho$ on $f^{-1}%
(C_{j}^{-})$ is equivalent to the action of $z\rightarrow e^{2\pi iq_{j}/p}$
on the unit cicle in complex plane.
\end{proof}

\begin{proof}
[Proof of the Theorem]Let $(\Sigma,\rho)$ be a topological realization of
$(S,h)$ given by the Lemma above. Using \cite{Ko} we can deform $\Sigma
/\left\langle \rho\right\rangle $ to an embedded surface $\Sigma^{\prime}$
that is conformally equivalent to $S/\left\langle h\right\rangle $, and since
the deformation can be chosen very close to $\Sigma/\left\langle
\rho\right\rangle $, we can consider that $\Sigma^{\prime}\cap OZ=\varnothing
$. If $f:\mathbb{R}^{3}\rightarrow\mathbb{R}^{3}/\left\langle \rho
\right\rangle =\mathbb{R}^{3}$, $f^{-1}(\Sigma^{\prime})$ is the surface that
we are looking for.
\end{proof}

\begin{remark}
It is possible to extend the results in this Section to automorphisms of Klein
surfaces with boundary orientable or not. The result obtained in that
situation is as follows: Let $f$ be an automorphism of prime order $p$ of a
bordered Klein surface. The automorphism $f$ is representable as a rotation if
and only if the rotation indices of fixed points of $f$ are all equal to
$2q\pi/p$, for some $q$ such that $0<q<p/2$.
\end{remark}

\section{Examples: visualizing two famous automorphisms of order seven.}

\subsection{The order seven automorphism of the Klein quartic}

In \cite{K2}, F. Klein introduces the a Riemann surface $S_{K}$ of genus $3$
with automophisms group isomorphic to $PSL(2,7)$. As the title of \cite{K2}
suggests the automorphisms of order seven of $S_{K}$ have a special importance
and Klein remarks that such an automorphism cannot be represented as the
restriction of a rotation on an embedded surface in the Euclidean space.

The surface $S_{K}$ is described as the quartic $x^{3}y+y^{3}z+z^{3}x=0$ with
the order seven automorphism $f$:

\begin{center}
$x\rightarrow\zeta x$, $y\rightarrow\zeta^{2}y$, $z\rightarrow\zeta^{4}z $,
where $\zeta=e^{2\pi i/7}$.
\end{center}

The orbit space $S_{K}/\left\langle f\right\rangle $ is the Riemann sphere and
$S_{K}\rightarrow S_{K}/\left\langle f\right\rangle $ is a cyclic covering
branched on three points, i. e. the automorphism $f$ has three fixed points:
$p_{1}$, $p_{2}$ and $p_{3}$. Let $D_{1}$, $D_{2}$ and $D_{3}$ be three
disjoint discs in $S_{K}$ around $p_{1},p_{2}$ and $p_{3}$ respectively and we
shall consider that $D_{i}$ is equivariant by the action of $f$. Assume that
the action of $f$ on $D_{1}$ is equivalent to the rotation $z\rightarrow\zeta
z$, on $D_{2}$ is equivalent to the rotation $z\rightarrow\zeta^{2}z$ and on
$D_{3}$ is equivalent to the rotation $z\rightarrow\zeta^{4}z$.

Now consider the bordered orientable Klein surface $B_{K}=S_{K}-$
$\overset{\circ}{D}_{2}\cup\overset{\circ}{D}_{3}$. Then $(B_{K},f)$ is
representable by a rotation, i. e., $B_{K}$ is isomorphic to a surface
$B_{K}^{\prime}$ that is embedded in the Euclidean space in such a way that
$f$ produces on $B_{K}^{\prime}$ the restriction of a rotation of order seven
$\rho$. The Figure 3 shows the surface $B_{K}^{\prime}$ and has been
constructed using the method in Section 2. Note that the surface in the Figure
3 is topologically equivalent to $(B_{K},f)$ and the rotation indices in the
fixed points and boundaries coincide.%

\begin{figure}
[ptb]
\begin{center}
\includegraphics[
height=3.2067in,
width=3.2508in
]%
{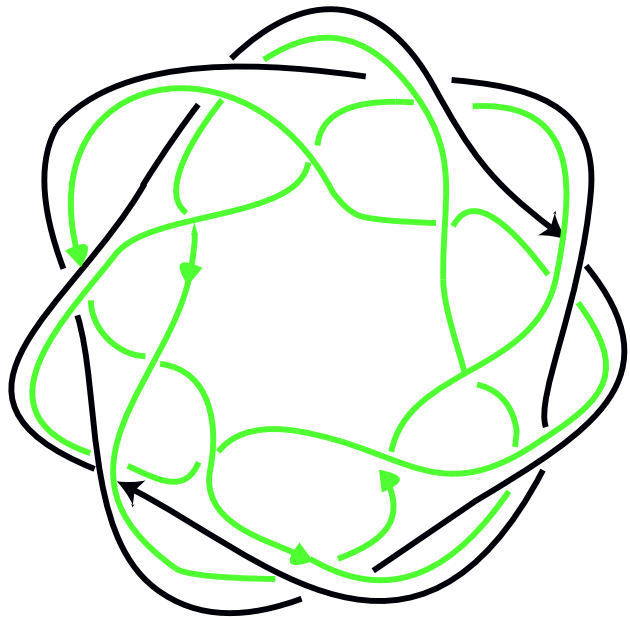}%
\caption{The order seven automorphism of the Klein quartic without two discs
embedded in the Euclidean space}%
\end{center}
\end{figure}

Consider $\mathbb{S}^{4}$ as two dimension 4 balls $B_{1}$, $B_{2}$ such that
$\partial B_{1}=\partial B_{2}$. The surface $B_{K}^{\prime}$ is inside the
tridimensional sphere $\partial B_{1}=\partial B_{2}$ in such a way that the
rotation giving the order seven automorphism is the restriction of a rotation
of $\mathbb{S}^{4}$. Let $C_{1}$ and $C_{2}$ be the boundary components of
$B_{K}$. The cone on $C_{1}$ from the center $c_{1}$ of $B_{1}$ and the cone
on $C_{2}$ from the center $c_{2}$ of $B_{2}$, produce a genus $3$ surface
$S_{K}^{\prime}$ immersed in $\mathbb{S}^{4}$. There is a rotation $r$ of
$\mathbb{S}^{4}$ leaving invariant $S_{K}^{\prime}$ and such that
$(S_{K}^{\prime},r)$ is topologically equivalent to $(S_{K},f)$.

Furthermore, since $S_{K}^{\prime}-\{c_{1},c_{2}\}/\left\langle r\right\rangle
$ is a sphere without two points then the conformal structure inherited from
$\mathbb{S}^{4}/\left\langle r\right\rangle $ is conformally equivalent to the
Riemann sphere without two points. Such structure of $S_{K}^{\prime}%
-\{c_{1},c_{2}\}/\left\langle r\right\rangle $ lifts in a unique equivariant
way to $S_{K}^{\prime}-\{c_{1},c_{2}\}$. Hence the surface $S_{K}^{\prime
}-\{c_{1},c_{2}\}$ with the conformal structure given by $\mathbb{S}^{4}$ is
conformally equivalent to the Klein quartic without two points, so in this
case is not necessary to deform the surface using \cite{Ko}.

\subsection{The order seven automorphism of the Wiman surface of genus $3$}

There is only one Riemann surface of genus $3$ and with an automorphism of
order $7$ different from the Klein quartic. This is the genus $3$ Wiman
surface $S_{A}$: $w^{2}=z^{7}-1$ \cite{W}. In this case the order seven
automorphism $f_{A}$ has also three fixed points: $p_{1}$, $p_{2}$ and $p_{3}%
$. Let $D_{1}$, $D_{2}$ and $D_{3}$ be three disjoint discs in $S_{A}$ around
$p_{1},p_{2}$ and $p_{3}$ respectively and $D_{i}$ is equivariant by the
action of $f_{A}$. The action of $f_{A}$ on $D_{1}$ is equivalent to the
rotation $z\rightarrow\zeta z$, on $D_{2}$ is equivalent to the rotation
$z\rightarrow\zeta z$ and on $D_{3}$ is equivalent to the rotation
$z\rightarrow\zeta^{5}z$. The surface $B_{A}=S_{A}-\overset{\circ}{D}_{3}$ is
homeomorphic to a surface $B_{A}^{\prime}$ that is embedded in the Euclidean
space in such a way that $f_{A}$ induces on $B_{A}^{\prime}$ the restriction
of a rotation: see the Figure 4. In fact $B_{A}$ is the Seifert surface of
minimal genus of the link of the singularity $w^{2}-z^{7}=0$, that is the
torus knot $(2,7)$.%

\begin{figure}
[ptb]
\begin{center}
\includegraphics[
height=4.3223in,
width=3.058in
]%
{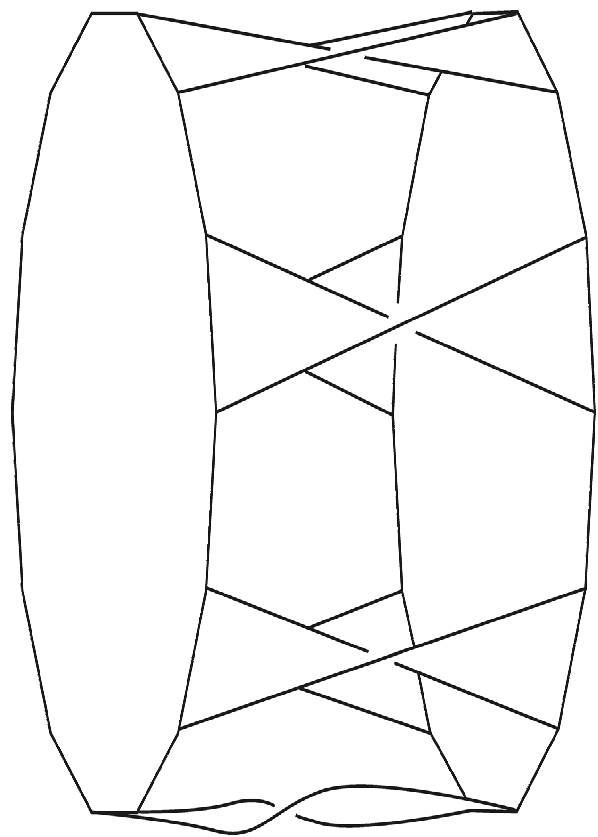}%
\caption{The order seven automorphism of the punctured Wiman surface embedded
in the Euclidean space}%
\end{center}
\end{figure}

Let $B$ be the unit ball of $\mathbb{R}^{4}$ with center in $0$. Consider
$B_{A}^{\prime}$ in the tridimensional sphere $\partial B$ in such a way that
the rotation giving the order seven automorphism is the restriction of a
rotation of $\mathbb{R}^{4}$. Let $C$ be the boundary of $B_{A}$. The cone on
$C$ from $0$, produces a genus $3$ surface $S_{A}^{\prime}$ immersed in
$\mathbb{R}^{4}$. There is a rotation $r$ of $\mathbb{R}^{4}$ leaving
invariant $S_{A}^{\prime}$ and such that $(S_{A}^{\prime},r)$ is topologically
equivalent to $(S_{A},f)$.

Furthermore, and since $S_{A}^{\prime}\rightarrow S_{A}^{\prime}/\left\langle
r\right\rangle $ has three branched points, the surface $S_{K}^{\prime}-\{0\}$
with the conformal structure induced by $\mathbb{R}^{4}$ is conformally
equivalent to the Wiman surface of genus $3$ without one point. Note that
there is an automorphism of order two $h$ in $S_{K}^{\prime}-\{0\}$ permuting
the two fixed points of the automorphism of order seven and with seven fixed
points. The automorphism $h$ can be represented also by a rotation indicating
that the surface $S_{K}$ is hyperelliptic and the product $h\circ r$ is the
Wiman automorphism of order $14$ of $S_{K}$.\bigskip

\textbf{Acknowlegments.} We wish to thank Le Fonds National Suisse de la
Recherche Scientifique for its support.

\end{document}